\documentclass[12pt]{amsart}
\usepackage{amsmath,amscd,latexsym,verbatim,amssymb}
\usepackage{times}       

  \newcommand{\resp}{{\it resp.} }
\newcommand{\cf}{{\it cf.} }
\newcommand{\ie}{{\it i.e.} }
\newcommand{\eg}{{\it e.g.} }

\newcommand{\loccit}{{\it loc. cit.} }

 \renewcommand{\d}{{\sm{\partial}}}

\newcommand{\Q}{\mathbb{Q}}

\newcommand{\C}{\mathbb{C}}

\newcommand{\sI}{{\mathcal{I}}}
\newcommand{\sK}{{\mathcal{K}}}
\newcommand{\sL}{{\mathcal{L}}}

\newcommand{\sM}{{\mathcal{M}}}
\newcommand{\sN}{{\mathcal{N}}}
\newcommand{\sO}{{\mathcal{O}}}
\newcommand{\sP}{{\mathcal{P}}}

\newcommand{\sR}{{\mathcal{R}}}
\newcommand{\sS}{{\mathcal{S}}}

\newcommand{\inj}{\hookrightarrow}

\newcommand{\surj}{\rightarrow\!\!\!\!\!\rightarrow}

\font\sm=cmr10 at 10pt
\font\smit= cmti10 at 10pt

 \newcommand{\End}{\operatorname{End}}
  \newcommand{\Hom}{\operatorname{Hom}}
  \newcommand{\car}{\operatorname{char}}

\newcommand{\Ker}{\operatorname{Ker}}

\newcommand{\rk}{\operatorname{rk}}

\newcommand{\im}{\operatorname{im}}
 \newcommand{\Spec}{\operatorname{Spec}}
\newcounter{spec}

\swapnumbers

\newtheorem{thm}{Theorem}[subsection]
\newtheorem{lemma}[thm]{Lemma}
\newtheorem{prop}[thm]{Proposition}

\newtheorem{cor}[thm]{Corollary}

\theoremstyle{definition}
 \newtheorem{defn}[thm]{Definition}
 
\newtheorem{ex}[thm]{Example}
\newtheorem{exs}[thm]{Examples}

\newtheorem{rem}[thm]{Remark}

\numberwithin{equation}{section}

\font\sm=cmr10 at10pt
 at12pt

\renewcommand{\qed}{\hfill $\square$\medskip}

\begin{document}

\title[Solution algebras and quasi-homogeneous varieties]{Solution algebras of differential equations and quasi-homogeneous varieties: a new differential Galois correspondence}

\author{Yves
Andr\'e}
 
  \address{D\'epartement de Math\'ematiques et Applications, \'Ecole Normale Sup\'erieure  \\ 
45 rue d'Ulm,  75230
  Paris Cedex 05\\France.}
\email{andre@dma.ens.fr}
   \date{\today}

  \begin{abstract}  We develop a new connection between Differential Algebra and Geometric Invariant Theory, based on an anti-equivalence of categories between \emph{solution algebras} associated to a linear differential equation  (\ie  differential algebras generated by finitely many polynomials in a fundamental set of solutions), and {\it affine quasi-homogeneous varieties} (over the constant field) for the differential Galois group of the equation. 
  
  Solution algebras can be associated to any connection over a smooth affine variety. It turns out that he spectrum of a solution algebra is an algebraic fiber space over the base variety, with quasi-homogeneous fiber. We discuss the relevance of this result to Transcendental Number Theory. 
  \end{abstract}
\maketitle

 \begin{sloppypar}

\section*{Introduction}

 \bigskip\subsection{Introduction}\label{su1.1} Let $\,K\,$ be a field endowed with a non-zero derivation  $\,\d$, with  algebraically closed constant field $\,C= \Ker \d\,$. 
Let $$ \ \phi(y) = \d^n y + a_{n-1} \d^{n-1}y  + \cdots + a_0 y = 0$$ be a linear differential equation with coefficients $a_i$ in  $K$, and let $\,y_0,\ldots, , y_{n-1}\,$ form a $C$-basis of solutions in some differential extension of $K$ with constant field $C$. 

\smallskip The Picard-Vessiot algebra of $\phi$ is the $K$-algebra generated by the derivatives $\d^j y_i $ and the inverse of the wronskian $ {\rm{det}}(\d^j y_i)$. It is the ring of coordinates of a principal homogeneous space over $K$ under the differential Galois group $G$ of $\phi$. Through Kolchin's work, this fact has been a source of motivation and applications in the early development of the theory of linear algebraic groups and their principal homogeneous spaces (\cf \cite[chap. VIII]{Bor}).

 \medskip  In this paper, we study the finitely generated differential subalgebras of a Picard-Vessiot algebra, which we call {\it solution algebras}.  
 
 Curiously,  traditional differential Galois theory has little to say about solution algebras beyond the Picard-Vessiot case - for instance about the algebraic relations between a single solution $y_0$ and its derivatives (a problem which occurs in transcendental number theory for instance, \cf \ref{su1.6}\footnote{after completion of this work, D. Bertrand pointed out to us the paper \cite{Com1} (\cf also \cite{Com2}), in which this problem is studied for generalized confluent hypergeometric differential equations.}). 
 
 The traditional differential Galois correspondence classifies differential {\it subfields} of the {\it fraction field} of the Picard-Vessiot algebra.  No such classification in terms of subgroups of the differential group $G$ exists at the level of differential subalgebras. 
 
 For instance, the Picard-Vessiot algebra $C(z)$-algebra $R'$ of the Airy equation $\frac{d^2y}{dz^2} = zy$ is the coordinate ring of $SL_2$, and the subalgebra $A$ generated by the logarithmic derivative of a single non-zero solution $y_0$ is a finitely generated differential subalgebra of the fraction field $Q(R')$ (not of $R'$); the fraction field of $A$ corresponds to a Borel subgroup $B$ of $SL_2$: one has $Q(R')^B= Q(A)$; but $(R')^B=C$, not $A$.   
    
   \medskip As we shall see,  the study of solution algebras involves finer notions from geometric invariant theory than just algebraic groups and torsors:  in fact, the whole theory of affine quasi-homogeneous varieties comes into play. 
   
 The differential Galois correspondence can be restored at the level of solution algebras in the form of an {\it anti-equivalence of categories between solution algebras as above and affine quasi-homogeneous $G$-varieties over $C$}.

  \smallskip After pioneering work by Grosshans, Luna, Popov, Vinberg and others in the seventies, the study of {\it quasi-homogeneous $G$-varieties}, \ie algebraic $G$-varieties with a dense $G$-orbit, has now become a rich and deep theory. The precise dictionary given below between the theory of affine quasi-homogeneous varieties and differential Galois theory should thus enrich considerably the latter, and may provide a source of motivation and applications for the former.
 We take advantage of this correspondence to study the algebraic structure of solution algebras (for instance, linear relations between solutions), with an eye towards transcendental number theory.

\section{Statement of the main results}

Our results take place in the general context of modules with connection over an affine basis\footnote{for a more geometric setting, see 6.5 (2).}, but in this introduction, we restrict ourselves to the context of differential modules over a differential ring (in the classical sense).

\subsection{Picard-Vessiot fields (reminder, \cf \cite{M}\cite{PS})} Let $(K, \d)$ be a differential field with algebraically closed constant field $C= K^\d$ of characteristic $0$. Let $K\langle \d\rangle$ denote the corresponding ring of differential operators. Let $ M $ be a differential module over $K$, that is, a $K\langle \d\rangle$-module of finite dimension $n$ over $K$ (for instance $M = K\langle \d\rangle/K\langle \d\rangle \phi,$ where $\phi$ is a differential operator as above). The finite direct sums of tensor products $M^{\otimes i}\otimes (M^\vee)^{\otimes j}$ and their subquotient differential modules form a tannakian category $\langle M\rangle^{\otimes}$ over $C$.
  
 A Picard-Vessiot field $K' $ for $M$ is a differential field extension of $K$ with constant field $C$, in which $M$ and its dual $M^\vee$ are solvable (\ie ${\rm{Sol}}(M, K'):= {\rm{Hom}}_{K\langle \d\rangle}(M, K')$ and ${\rm{Sol}}(M^\vee, K')$ have dimension $n$ over $C$), and which is minimal for this property. Such a differential field exists and is unique up to non-unique isomorphism. The differential Galois group of $M$, 
 $$G= {\rm Aut}_\d\, K'/K, $$ is a linear algebraic group over $C$ which acts faithfully on $\,{\rm{Sol}}(M, K')$.
 
The differential Galois correspondence is an order-reversing bijection between intermediate differential extensions $K\subset L\subset K'$ and closed subgroups $H< G\,$, given by $H = {\rm Aut}_\d\, K'/L$ and $L= (K')^H$. One has $\,{\rm{tr.deg}}_KL= \dim G-\dim H$.

\subsection{Solution fields} 

\begin{defn} A {\emph{solution field}} $(L,\d)$ for $M$ is a differential field extension of $(K,\d)$ with constant field $L^\d = C$, which is generated by the image of a $K\langle \d\rangle$-morphism $v: M\to L$. 
\end{defn}
    
     \smallskip   In the next theorem, ``solution field" means  ``solution field for some $N\in  \langle M\rangle^{\otimes}$". For instance, the Picard-Vessiot field $K'$ is a solution field for $M^n \oplus (M^\vee)^n$.
    
    \begin{thm}\label{T1.1} \begin{enumerate}      \item  Any solution field $L$ embeds as a differential subfield of the Picard-Vessiot field $K'$.
    \smallskip\item Conversely, an intermediate differential field $K\subset L\subset K'$ is a solution field if and only if the corresponding subgroup $H<G$ is {\emph{observable}} (\ie $G/H$ is quasi-affine). In fact, $H$ is the isotropy group of any solution $v\in {\rm{Sol}}(N, K')$ whose image generates $L$.
     \smallskip\item For any solution field $L=(K')^H$,  $ \, {\rm Aut}_\d\, L/K = N_G(H)/H.$
    \end{enumerate}
    \end{thm}

\subsection{Picard-Vessiot algebras}  Even though this result is formulated in terms of traditional differential Galois theory of differential fields, our proof  uses the generalized differential Galois theory for differential rings developed in \cite{A3} (working over differential rings rather than fields is natural, useful, and sometimes necessary in some contexts).

 \smallskip Let $(R, \d)$ be a differential ring with constant field $C$. We assume that $(R, \d)$ is simple, \ie has no non-zero proper differential ideal. It is then known that $R$ is an integral domain, and we denote by $K$ its quotient field. 
  
  Let $M$ be a differential module of finite type over $R$. It can be shown that $M$ is projective, and so are all the finite direct sums of tensor products $M^{\otimes i}\otimes (M^\vee)^{\otimes j}$ and their subquotient differential modules, which form a tannakian category $\langle M\rangle^{\otimes}$ over $C$ (equivalent to $\langle M_K\rangle^{\otimes}$),  \cf \ref{T2.1} below  (instead of $M^{\otimes i}\otimes (M^\vee)^{\otimes j}$, one may consider $M^{\otimes i}\otimes (\det\, M)^{\otimes -j}$, where $\det\, M$ denotes the top exterior power). 
  
 The {Picard-Vessiot algebra} $R'$ for $M$  is the $R$-subalgebra of the Picard-Vessiot field $K'$ for $M_K$  generated by  $\langle M, {\rm{Sol}}(M, K')\rangle$ and $\langle M^\vee,  {\rm{Sol}}(M^\vee, K')\rangle$, its spectrum is a torsor under $G_R$, and $G= {\rm{Aut}}_\d(R'/R)$.

\subsection{Solution algebras}  
    \begin{defn}\label{df2} A {\emph{solution algebra}} $(S,\d)$ for $M$ is a differential $R$-algebra without zero-divisor, whose quotient field has constant field $C$, and which is generated by the image of a $R\langle \d\rangle$-morphism $v: M\to S$. 
\end{defn}
   
 The link with the previous definition is the following (\cf \ref{L4.2}): a differential algebra extension $S/R$ is a solution algebra for $M$ if and only if it is a finitely generated $R$-algebra without zero-divisor and its quotient field $L$ is a solution field for $M_K$; any solution field $L$ for $M_K$ is the quotient field of a solution algebra for $M$.

     \smallskip   In the next theorem,  ``solution algebra" means ``solution algebra for some {$N\in  \langle M\rangle^{\otimes}\,$}".  
         
   \begin{thm}\label{T1.2} \begin{enumerate}    \item Any differential finitely generated sub-$R$-algebra of the Picard-Vessiot algebra $R'$  is a solution algebra.
       \smallskip\item  If $S$ is a solution algebra, then for any embedding of the quotient field $L$ of $S$ into $K'$, $S$ is contained in the Picard-Vessiot algebra $R'$.
                \smallskip\item For any solution algebra $S$ generated by a solution $v$, ${\rm{Spec}}\, (S_{K'})^\d\,$ is the closure $\,\overline{G.v}\,$ of the orbit $\;G. v   \,\subset \,{\rm{Sol}}(M, K')$. This provides an anti-{\emph{equivalence}} of categories between solution algebras and affine quasi-homogeneous $G$-varieties.
         \smallskip\item If $H<G$ is observable,  $(R')^H$ is a solution algebra if and only if $H $ is Grosshans (\ie  $C[G/H]$ is finitely generated).
   \smallskip\item A solution algebra $S$ is simple (as a differential ring) if and only if it is generated by a solution $v$ for which the orbit $\,G.v\,$ is closed. In that case, $S=  (R')^H$.
      \smallskip\item A solution field $L$ is the quotient field of a {{unique}} solution algebra $S$ if and only if the image $\bar H$  of $H$ in the reductive quotient $\bar G$ of $G$ is reductive and $N_{\bar G}(\bar H)/\bar H$ is finite. In that case,  $S$ is simple.
      \smallskip\item Assume that $R$ is finitely generated over $C$.  Then, locally for the \'etale topology on ${\rm{Spec}}\,R$, the spectrum of a solution algebra  $S$ generated by a solution $v$ is isomorphic to $(\overline{G.v})_R\,$ (in particular,  it is an algebraic fiber bundle over $\,{\rm{Spec}}\,R$).
    \end{enumerate}
    \end{thm}

    \subsection{From affine quasi-homogeneous varieties to differential modules}\label{su1.6}  On combining the previous theorem with the constructive solution \cite{MS} of inverse differential Galois problem and the triviality of torsors over $C[z]$ under (pull-back of) reductive groups over $C$  \cite{RR}, one obtains the following 
   
   \begin{thm}\label{T1.3} 
       \begin{enumerate}   \item The differential Galois group $G$ of any semisimple differential module $M$ over $(C[z], \frac{d}{dz})$ is connected reductive, and the spectrum of any solution algebra $S$ for 
$ M$ satisfies $\,\Spec\, S\cong Z_{C[z]}\,$ for some affine quasi-homogeneous $G$-variety $Z$ over $C$. 
    \smallskip   \item  Conversely, to any connected reductive group $G$ over $C$ and any affine quasi-homogeneous $G$-variety $Z$,  one can attach in a constructive way a semisimple differential module $M$ over $ C[z] $ with differential Galois group $G$, and a solution algebra $S$ for $M$ such that $\,\Spec\, S\cong Z_{C[z]}$.  \end{enumerate}   \end{thm}
   
     Using work by Arzhantsev and Timashev \cite{ArT2} on quasi-homogeneous varieties with infinitely many orbits, one can construct in this way solution algebras over $C[z]$ or $C(z)$ which {\it admit infinitely many quotients which are solution algebras} (\cf Remark \ref{RR3}): this occurs for any connected reductive differential Galois group $G$, taking for isotropy group $H$ the unipotent radical of any non-minimal parabolic subgroup of $G$. 
  
  On the other hand, the negative solution of Hilbert's XIV$^{th}$ problem provides observable subgroups $H$ which are not Grosshans, and one can construct in that way (\cf  \ref{PT1.2} (4)) {\it integrally closed solution algebras  $S$ over $ C[z] $ or $C(z)$  whose maximal localization $Q(S)\cap R'$ in the Picard-Vessiot algebra is not finitely generated}.
  
 The classification of solution algebras is an arduous task, even over $C[z]$ or $C(z)$: for instance, $C(z)$-algebras generated by polynomials in solutions of the Airy equation, and their derivatives,  correspond to affine quasi-homogeneous $SL(2)$-varieties; the normal ones are classified by discrete invariants, but the non-normal ones may form {\it continuous families} \cite{Ba}.

        \subsection{Homogeneous relations}   
 Let $S$ be a solution algebra generated by a solution $v: M\to S$. Then $v$ extends to a surjective homomorphism of differential rings $v^{\cdot}:{\rm{Sym}}^{\cdot} M \to S$.     
  Let ${\tilde S}$ be the quotient of $\, {\rm{Sym}}^{\cdot} M$ by the (differential) ideal generated by homogeneous relations with respect to $M$  in ${\rm{Ker}}\, v^{\cdot}$. 
      
\begin{thm}\label{T1.4}  \begin{enumerate} \item $S$ is homogeneous  (\ie $S= {\tilde S}$) if and only if there exist $g\in G$ and $\lambda\in C$, not a root of unity, such that $g.v=\lambda v$.

 \smallskip \noindent Assume that $R$ is finitely generated over $C$. Then

\smallskip   \item   $\,{\rm{Proj}}\,{\tilde S}$ is an algebraic fiber bundle over $\,{\rm{Spec}}\,R\,$ (locally trivial for the \'etale topology).  

\smallskip  \item $K$ is algebraically closed in $L\,\Leftrightarrow  \, $  all fibers of $\,\Spec   S $ are integral $\Rightarrow  \, $ all fibers of $\,{\rm{Proj}}\,{\tilde S}$ are integral. 
  \end{enumerate}   
     \end{thm}  
 
\subsection{Relevance to transcendental number theory}\label{su1.6}

Let us consider a solution $y = \sum a_m z^m\in \bar\Q[[z]]$ of a linear differential equation $\phi(y)=0$ of order $n$ with coefficients in $R= \bar\Q[z, \frac{1}{T(z)}]$.   

\begin{cor}\label{C1.5} Assume that $\bar\Q(z)$ is algebraically closed in $\,\bar\Q(z, y,   \ldots, y^{(n-1)} ={\frac{d^{n-1}y}{dz^{n-1}}})$. Let $\xi\in \bar\Q^\ast$ be in the domain of convergence of $y$, and not a zero of the polynomial $\, T$.  
 
Assume that the transcendence degree (\resp homogeneous  transcendence  degree) of  $\bar\Q[y(\xi), , \ldots, y^{(n-1)}(\xi)]$ over $\bar\Q$ equals the  transcendence degree (\resp homogeneous  transcendence  degree) of  $\,\bar\Q(z)[  y,   \ldots, y^{(n-1)}] \,$  over $\,\bar\Q(z)$.

Then any polynomial relation (\resp homogeneous polynomial relation) with coefficients in $\bar\Q$ between $y(\xi), \ldots, y^{(n-1)}(\xi) $ is the specialization at $\xi$ of a polynomial relation (\resp homogeneous polynomial relation of the same degree) with coefficients in $\,R\,$ between $y , \ldots, y^{(n-1)}$.

In particular, if the functions $y,   \ldots, y^{(n-1)}$ are linearly independent over $\bar\Q(z)$, their values $y(\xi), \ldots, y^{(n-1)}(\xi) $ are linearly independent over $\bar\Q$.\end{cor} 

Indeed, since $\bar\Q(z)$ is algebraically closed in the solution field $L= \bar\Q(z, y,   \ldots, y^{(n-1)})$, the fiber of   $\,\Spec { S}$ (\resp $\,{\rm{Proj}}\,{\tilde S}$) at $\xi$ is integral according to \ref{T1.4} (3).  It contains the affine (\resp projective) variety defined by the (\resp homogeneous) polynomial relations with coefficients in $\bar\Q$ beween $y(\xi), \ldots, y^{(n-1)}(\xi)$. Hence these $\bar\Q$-varieties coincide if they have the same dimension.

\smallskip The assumptions of the corollary are notably satisfied when $y$ is an $E$-function (for instance $y = \sin z$), or more generally an arithmetic Gevrey series of negative rational order $s$ \cite{A2}, \ie when the absolute logarithmic height of $(a_1. 1!^{-s},\ldots, a_m.m!^{-s})$  grows at most linearly with $m$.  In that case, $L$ consists of meromorphic functions on $\C$, hence $\bar\Q(z)$ is algebraically closed in $L$, and the condition about transcendence degrees is essentially the classical Siegel-Shidlovsky theorem, which can be also derived rather directly from the fact (proven in \cite{A1}) that differential operators $\phi$ of minimal order annihilating such series $y$ have no non trivial singularities at finite distance. 

In \cite{Be}, Beukers uses this fact to deduce, for $E$-functions, the conclusion of the above corollary from the Siegel-Shidlovsky theorem (answering an old question of Lang \cite[p. 100]{La}). However, as we have seen (\cf also \ref{R6.1}), such a deduction actually follows from general results of (generalized) differential Galois theory, independently of \cite{A1}.

\bigskip  
   \section{Generalized Picard-Vessiot theory. A reminder and some complements to \cite{A3}}     
 \subsection{}  In order to extend the scope of our results and cover the case of simultaneous action of several derivations, and connections on higher dimensional varieties, we shall work with generalized differential rings as in \cite{A3}, which keeps the spirit of classical differential algebra.
 
 \smallskip  Let $\sR =(R, \; d: R\to \Omega)$ be a {\it generalized differential ring}, \ie the data of a commutative ring $R$ and a derivation $d: R\to \Omega$ to a $R$-module $\Omega$, which we always assume to be {\it projective of finite rank}  (the classical notion of differential ring corresponds to the case $\Omega = R$).
       We denote by $C= {\rm{Ker}}\, d\,$ the ring of constants.

  \smallskip An {\it extension} $\sS/\sR$ consists of a ring extension $S/R$ together with an extension $S\to \Omega\otimes_R S$ of the derivation $d$.
  
  A {\it differential module} $\sM =(M, \nabla)$ over $\sR$ is an $R$-module $M$ with a connection $\nabla$, \ie an additive map $M\to M\otimes_R \Omega$ satisfying the Leibniz rule. We write $\sM^\nabla$ for the kernel of $\nabla\,$ (a $C$-module).
   
   A {\it differential ideal} $\sI$ is a differential submodule of $\sR$ (equivalently, the data of an ideal $I$ of $R$ such that $\langle  \Omega^\vee, dI\rangle \subset I$. 
   
   One says that 
  $\sR$ is {\it simple} if it has no non-zero proper differential ideal. 
  
\begin{exs} If $X$ is an affine smooth geometrically connected variety over a field $C$ and $\Omega= \Gamma(X, \Omega^1_{X/C})$, then $(\mathcal O(X),d)$ is a simple differential ring.
 
 Local rings of complex analytic manifolds are simple differential rings.\end{exs}
  
  \begin{lemma}\label{simple}    Let us assume that $\sR$ is simple. Then  \begin{enumerate}  
  \item $C$ is a field. 
 
 \medskip\noindent Assume that $\car C=0$. Then
 
 \smallskip  \item  $R$ is an integral domain.
 
    \smallskip\item  There is a unique extension of $d$ to the quotient field $K$ of $R$ which defines a differential extension $\sK/\sR$, with constant ring $C$. 
  
 \end{enumerate}        \end{lemma}
  
  \begin{proof} For items (1) and (3), see \cite[2.1.3.5]{A3}.    
  The proof of (2) given in \cite[Lemma 1.17]{PS} in the case $\Omega=A$ extends to the general case: one first shows that every zero-divisor $a\in R$ is nilpotent (considering the differential ideal of elements $b$ such that $a^m b=0$ for some $m$); then that the nilradical of $R$ is a differential ideal (the image by any $\d\in \Omega^\vee$ of a nilpotent element is a zero-divisor).   \end{proof}
  
  \begin{lemma}\label{L2.1} Let  $\sM=(M,\nabla)$ be a differential module over a simple differential ring $\sR$. Then the natural morphism $M^\nabla\otimes_C R\to M$ is injective.
   \end{lemma}
   \begin{proof}   \cf \cite[3.1.2.1]{A3}.
  \end{proof}

\begin{cor}\label{C1.3} For any field extension $C'/C$, $\sR_{C'}$ is simple.
\end{cor}  
\begin{proof}
 Let $\sI \subset \sR_{C'}$ be a proper differential ideal, and let $\sM= \sR_{C'}/\sI$. Then  $\sM^\nabla$ contains $C'$, and the natural projection $\sR_{C'}\to \sM$ can be written as the composition
 $\sR_{C'}\inj M^\nabla\otimes_C \sR \to \sM $, and is injective by the previous lemma, whence  $ \sI= 0$.\end{proof}

   \subsection{}\label{su2.2} In algebraic geometry, it is well-known that coherent modules with integrable connection over a smooth basis are locally free. It is less known that the integrability condition is superfluous. An abstract explanation is provided by the following theorem.
   
 \smallskip We assume henceforth that $\sR$ is {\it simple} and $\car C=0$, and denote by $\sK=(K,d)$ its quotient field (considered as a differential extension of $\sR$). 
   
       \begin{thm}\label{T2.1}  Let $\sM$ be a differential module over $\sR$.
      Assume that the underlying $R$-module $M$ is finitely generated.     \begin{enumerate}  \item Then $M$ is projective. The same holds for any subquotient of $\sM$. 
   \smallskip \item  The finite direct sums of tensor products $\sM^{\otimes i}\otimes (\sM^\vee)^{\otimes j}$ and their subquotient differential modules form a tannakian category $\langle \sM\rangle^{\otimes}$ over $C$, and the natural $\otimes$-functor $\langle \sM\rangle^{\otimes} \to \langle \sM_\sK\rangle^{\otimes}$ is an equivalence. 
  \end{enumerate}     \end{thm}
   
     \begin{proof}    (1)    $M$ is an $R$-lattice in the vector space $M_K$ in the sense of Bourbaki \cite[VII.4.1]{Bou}, \ie a sub-$R$-module which spans $M_K$ and is contained in a finitely generated $R$-submodule.
    According to \loccit, for any $R$-lattice $N$, $M\otimes_R N$ is a lattice in $M_K\otimes_K N_K$ and $\Hom_R(M,N)$ is a lattice in $Hom_K(M_K,N_K)$ (in particular the dual $M^\vee$ is a lattice in $(M_K)^\vee$). 
    
    It follows that if $\sN$ is another differential module, of finite type over $R$ (or more generally such that $N$ is a lattice in $N_K$), the natural $C$-linear map $\Hom (\sM, \sN)\to \Hom(\sM_\sK, \sN_\sK)$ is injective. It is surjective as well: if $f\in \Hom(\sM_\sK, \sN_\sK)$, $f(\sM)$ is an $\sR$-differential submodule of $\sN_\sK$ and the quotient $f(\sM)/(f(\sM)\cap \sN)$ is an $\sR$-differential module, finitely generated and torsion over $R$, Its annihilator is a non-zero differential ideal in $R$. Since $\sR$ is simple, we conclude that $f(\sM)/(f(\sM)\cap \sN)=0$, hence $f\in  \Hom (\sM, \sN)$.

    In particular the canonical coevaluation morphism $\eta_K: \,\sK \to \sM_\sK \otimes \sM_\sK^\vee$ comes from a coevaluation morphism $\eta:\, \sR\to \sM \otimes \sM^\vee.$ On the other hand, one has the evaluation morphism $\varepsilon:\,\sM^\vee \otimes \sM \to  \sR$, and the equation 
    $\, (1_\sM\otimes \varepsilon)\circ (\eta\otimes 1_\sM)=1_\sM\,$ holds since it holds after tensoring with $\sK$, taking into account the previous observation.  This shows that $M$ is projective.
    
    Any quotient of $\sM$ is again finitely generated over $R$, hence projective. And so is any subobject, viewed as the kernel of a quotient morphism.  
    
 \smallskip   (2)  The finite direct sums of tensor products $M^{\otimes i}\otimes (M^\vee)^{\otimes j}$ and their subquotient differential modules form an abelian $C$-linear $\otimes$-category $\langle \sM\rangle^{\otimes}$ with unit $\sR$, and $\End\, \sR= C$. By item (1), this is a rigid $\otimes$-category. The forgetful functor 
    $$\vartheta:  \langle \sM\rangle \to {\rm{Proj}}_R,\;\, \sN \mapsto N$$ is a fiber functor.  Hence $\langle \sM\rangle^{\otimes}$ is tannakian over $C$. We have already shown that the $\otimes$-functor  $\langle \sM\rangle^{\otimes} \to \langle \sM_\sK\rangle^{\otimes}$ is fully faithful.  It is essentially surjective because given $\sN\in \langle \sM\rangle^{\otimes}$,  every subobject $\sP$ in $\langle \sM_\sK\rangle^{\otimes}$ of $\sN_K$ comes  from a subobject of $\sN$ (with underlying $R$-module $ N\cap  P$). \end{proof}

      \subsection{ }\label{su2.3} We assume henceforth that $C$ is {\it algebraically closed of characteristic $0$}. It follows that $\langle \sM\rangle^{\otimes} $ admits a fiber functor $$\omega:  \langle \sM \rangle^{\otimes} \to {\rm{Vec}}_C ,$$  which  is unique up to non-unique isomorphism (if $R$ is finitely generated over $C$, one may take $ \omega = \vartheta_x =$ the fiber at any closed point $x$ of $\Spec\,R$, \ie the reduction modulo any maximal ideal of $R$). 
      
      The automorphism group scheme of $\omega$ is the {\it differential Galois group of $\sM$} (``pointed at $\omega$")
       $$G =  {\rm{Gal}}\,(\sM,\omega) = {\bf{Aut}}^\otimes\,\omega,$$
    a closed subgroup of $GL(\omega(\sM))$, and one has equivalences of tannakian categories  $\langle \sM \rangle^{\otimes}\cong \langle \sM_\sK\rangle^{\otimes} \cong {\rm{Rep}}\, G$. In particular, $\sM$ is semisimple if and only if the faithful $G$-module is semisimple, which is equivalent to: $G$ is reductive (since $\car C=0$).

    The isomorphism scheme $$\Sigma  =  {\bf{Iso}}^\otimes\,(\omega\otimes_C R,\vartheta) $$ is a torsor under the right action of $G_R$ (the {\it torsor of solutions} of $\sM$). 
     
      \subsection{ }\label{su2.4} A {\it solution} of $\sM$ in a differential extension $\sS/\sR$ is a morphism of differential modules $ \sM \stackrel{v}{\to}  \sS$ over  $\sR$. Since $M$ is projective of finite rank, this is the same as an element $\,v\in (\sM^\vee\otimes_\sR\sS)^\nabla$.   
      
    We say that $\sM$ is {\it solvable} in $\sS$ if the solutions of $\sM$ in $\sS$ generate $Hom_R(M,S)$ over $S$. Assume that $\sS$ is simple with constant field $C'$. Then, by Lemma \ref{L2.1}, $\sM$ is solvable in $\sS$ if and only if $(\sM^\vee_\sS)^\nabla\otimes_{C'} \sS \cong \sM^\vee_\sS$.
          If moreover $S$ is faithfully flat over $R$,  and both $\sM$ and $\sM^\vee$ are solvable in $\sS$ (equivalently:  $\sM$ and $(\det M)^\vee$ are solvable in $\sS$), then any $\sN \in  \langle \sM\rangle^{\otimes} $ is solvable in $\sS$ and  $ \omega_{\sS}:= (- \otimes_{\sR} \sS)^\nabla$ is a fiber functor on $\langle \sM\rangle^{\otimes} $ with values in ${\rm{Vec}}_{C'}$ (\cf \cite[3.1.3.2]{A3}).

  \smallskip    A {\it Picard-Vessiot algebra} $\,\sR'\,$ for $\sM$ is a faithfully flat simple differential extension of $\sR$ with constant field $C$ in which $\sM$ and $\sM^\vee$ are solvable, and which is minimal for these properties (which amounts to saying that $S$ is generated by $\langle M, \omega(\sM^\vee)\rangle$ and $\langle M^\vee, \omega(\sM)\rangle$). 
       
      Starting with a fiber functor $\omega$, there is a canonical structure of differential ring on $\sO(\Sigma)$ which makes it a Picard-Vessiot  algebra for $\sM$, and $\omega$ is canonically isomorphic to $\omega_{\sR'}$ (\cf \cite[3.4.2.1]{A3}). Any Picard-Vessiot  algebra for $\sM$ arises in this way up to isomorphism. One has 
      $$G = {\bf{Aut}}\,\sR'/\sR,$$ an equality compatible with the $G$-action on $\omega(\sM)$ in the pairing $ M^\vee \otimes_C \omega(\sM) \to  R'$.  For all this, we refer to \cite[\S 3.2, 3.4]{A3}.
      
      \begin{rem}\label{RR1} It is worth pointing out that we haven't assumed any finiteness condition on $R$, {\it nor any integrability condition on $\sM$}. At first, it might seem strange that a non-integrable connection is  solvable in some differential extension $\sR'/\sR$. This is discussed in detail in \cite[3.1.3.3]{A3}: the point is that for two commuting derivations $D_1, D_2\in \Omega^\vee$ (viewed as derivations of $A$), the eventuality that $\nabla_{D_1}$ and $\nabla_{D_2}$ do not commute is no obstruction for solvability in a  differential extension $R'$ in which the extension of $D_1$ and $D_2$ may not commute any longer.    \end{rem}
  
          \begin{rem}\label{RR2} (On the triviality of $\Sigma$). From Lemma \ref{simple} and the fact that $\sO(\Sigma)$ is a simple differential ring, it follows that $\Sigma$ is integral. In general, this torsor is non trivial, since the differential Galois group $G$ need not be connected. 
          
          However, {\it when $G$ is connected, and when $\sR$ is any localization of  $C[z]$} (viewed as a differential ring in the standard way), then {\it $\Sigma$ is a trivial torsor under $G_R$}: this follows from the triviality of torsors over open subsets of the affine $C$-line, under (pull-back of) connected linear algebraic $C$-groups, \cf \cite[prop. 5]{P}.    
      \end{rem}

      \begin{lemma}\label{L2.5} \begin{enumerate} \item For any field extension $C' /C$, the differential Galois group of $\sM_{C'}$ is $G_{C'}$, and $\sR'_{C'}$ is a Picard-Vessiot algebra for  $\sM_{C'}$.     \item $\sR'_\sK $ is the Picard-Vessiot algebra for $\sM_\sK$.\end{enumerate} \end{lemma} 
             
      \begin{proof} (1) Note that $\sR_{C'}$ and $\sR'_{C'}$ are simple with constant field $C'$ by Corollary \ref{C1.3}. On the other hand, $\sM_{C'}$ and its dual are solvable in $\sR'_{C'}$; and $\sR'_{C'}$ is generated by $\langle M^\vee_{C'}, (\sM_{\sR'_{C'}})^\nabla\rangle$ and $\langle M_{C'}, (\sM^\vee_{\sR'_{C'}})^\nabla\rangle$. Hence $\sR'_{C'}$ is a Picard-Vessiot algebra for  $\sM_{C'}$. Hence the torsor of solutions of $\sM_{C'}$ is $\Sigma_{C'}$, its right automorphism group is $G_{R_{C'}}$, and one concludes that the differential Galois group is $G_{C'}$.    
      
      (2) follows from the equivalence of categories established in item (2) of the previous theorem.  \end{proof} 
      
   \subsection{}\label{su2.5}  We still denote by $\omega$ the equivalence of ind-tannakian categories 
  $$ \omega = ( - \otimes_R R')^\nabla:\; {\rm{Ind}}\, \langle \sM\rangle^{\otimes} \to {\rm{Ind}}\, {\rm{Rep}}\, G.$$
  Note that ${\rm{Ind}}\, {\rm{Rep}}\, G$ is nothing but the category of rational $G$-modules, \ie  $C$-vector spaces on which $G$ acts as a group of automorphisms, and which are sums of finite-dimensional $G$-stable subspaces on which the given action of $G$ is by some rational representation \cf \eg \cite[p. 7]{G}. 
  For any $\sN\in {\rm{Ind}}\, \langle \sM\rangle^{\otimes}$, there is a canonical isomorphism of $\sR'$-differential modules  
  \begin{equation}\label{iso}  \omega(\sN)\otimes_C \sR' \stackrel{\sim}{\to} \sN\otimes_R \sR' \end{equation} (coming from the canonical $\sR'$-point of $\Sigma$).
  Since $R'$ is faithfully flat over $R$, we conclude that

\begin{cor}\label{C1.4} For any object $\sN$ in ${\rm{Ind}}\, \langle \sM\rangle^{\otimes}$, the underlying $R$-module $N$ is faithfully flat. \qed
\end{cor}
 
  Via $\omega$, differential algebra extensions of $\sR$ in ${\rm{Ind}}\, \langle \sM\rangle^{\otimes}$   correspond to rational $G$-algebras (for instance $\sR'$ correspond to $C[G]$ with $G$-action by left translations), and their differential ideals correspond to $G$-ideals.  
  
  \begin{cor}\label{C1.5}  Assume that $R$ is finitely generated over $C$. Let $\sS\in {\rm{Ind}}\, \langle \sM\rangle^{\otimes}\,$ be a differential algebra extension of  $\sR$. Then locally for the \'etale topology on $\Spec R$,  $\, \Spec\,S$ is isomorphic to $\Spec \omega(\sS)\times_C R$.   \end{cor}   
 
   \begin{proof} By \eqref{iso}, $S$ and $\omega(\sS)_R$ become isomorphic after smooth surjective base change $\Spec R'\to \Spec R$, hence after \'etale surjective base change since $\Spec R'= \Sigma\to \Spec R$ is smooth surjective (\cf [EGAIV, 17.6.3]). \end{proof}

         \section{Solution algebras and affine quasi-homogeneous varieties} 
     
     Here again, $\sR$ is a simple (generalized) differential ring with algebraically closed field of constants $C$ of characteristic $0$, $\sK$ is its quotient field, and $\sM$ is a finitely generated differential module.
     
     \subsection{} Let $\sS/\sR$ be a differential extension. 
     
          \begin{defn}\label{D3.1}  $\sS$ is a {\emph{solution algebra}} for $\sM$ if 
     \begin{enumerate} 
     \item $S$ is a domain, \smallskip\item the constant field of its quotient field $L$ (viewed as a differential extension $\sL$ of $\sK$) is $C$,
     \smallskip\item there is a solution $v$ of $\sM$ in $\sS$ (\ie a morphism $ \sM \stackrel{v}{\to}  \sS$ of differential modules over $\sR$)   such that the image of $v$ generates the $R$-algebra $S$. 
      \end{enumerate}
      A solution algebra for $\langle \sM\rangle^{\otimes}$ is a solution algebra for some $\sN \in \langle \sM\rangle^{\otimes}$.
     \end{defn} 

\begin{ex}  A Picard-Vessiot algebra $\sR'$ for $\sM$ is a solution algebra for $\sM^r \oplus (\sM^\vee)^r$, with $r= \rk M$ (the solution $v$ being given by $(v_1,\ldots, v_r, v_1^\vee, \ldots, v_r^\vee)$, where $(v_1,\ldots, v_r)$ is a $C$-basis of solutions of $\sM$ in $\sR'$ and $(v_1^\vee, \ldots, v_r^\vee)$ is the dual basis). \end{ex}
       
     \begin{rem} Condition (2) is stronger than requiring that the constant ring of $\sS$ is $C$. For instance, if $\sR= (C[z], d=\frac{d}{dz}),\, \sM= (C[z]^2, \nabla= d - {\rm{diag}}(1,2)),\;S= C[x,y,z],$ with $ \,dx=x, \, dy=2y$, and $v$ maps the canonical basis of $M$ to $(x,y)$, then the constant ring of $\sS$ is $C$, but the constant field of its quotient field is $C(\frac{x^2}{y})$, so that $\sS$ is not a solution algebra for $\sM$ in the sense of Definition \ref{D3.1} (but its quotient by the differential ideal generated by $y-x^2$ is a solution algebra for $\sM$).  
     \end{rem} 
     
         \begin{ex}  If $\Omega= R$ and $\sM \cong \sR/ \sR. \phi$ is a cyclic differential module, then a solution algebra for $\sM$ is the differential $R$-algebra generated by a solution of $\phi$ (in some differential extension field with constant field $C$).
            \end{ex} 
            
     \begin{prop}\label{P3.1}  Any solution algebra for $\langle \sM\rangle^{\otimes}$ belongs to ${\rm{Ind}}\langle \sM\rangle^{\otimes} $,  hence is faithfully flat over $R$.   \end{prop} 
     
         \begin{proof} The morphism $v: \sN \to  \sS$ extends to a morphism $\, v^{\bf \cdot}:  {\rm{Sym}}^{\bf \cdot}\, \sN \to \sS\,$ which is surjective by item (3) of Definition \ref{D3.1}, hence $\sS\in {\rm{Ind}}\langle \sM\rangle^{\otimes} $.  Faithful flatness over $\,R\,$ follows, due to Corollary 2.5.1. \end{proof}
        
        \smallskip We fix a fiber functor $\,\omega: \langle \sM\rangle^{\otimes}\to {\rm{Vec}}_C$. Let $\,G\subset GL(\omega(V))$ be the differential Galois group of $\sM$, and let $\sR'$ be the Picard-Vessiot algebra of $\sM$, so that $\;R'= \sO(\Sigma), $ and $ \omega$ is canonically isomorphic to $(-\otimes_\sR \sR')^\nabla.$
        \smallskip
        
     \begin{prop}\label{T3.1}   
      \begin{enumerate}   \item Any solution algebra $\sS$ for $\langle \sM\rangle^{\otimes} $ embeds as a differential sub-extension of $\sR'/\sR$.
     \smallskip   \item Conversely, any differential sub-extension $\sS$ of $\sR'/\sR$ which is finitely generated over $R$ is a solution algebra for $\, \langle \sM\rangle^{\otimes}  $.
       \smallskip\item Given $\sN\in  \langle \sM\rangle^{\otimes} ,$ $\sS\mapsto \sS_\sK, \; \sS_K\mapsto \sS_\sK \cap \sR'$ are inverse bijections between solution algebras for $\sN$ in $\sR'$ and solution algebras for $\sN_{\sK}$ in $\sR'_\sK$.     \end{enumerate}
     \end{prop}
     
     \begin{proof} (1) Since the Picard-Vessiot algebra of $\sN$ embeds in $\sR'$, it suffices to consider the case $\sN=\sM$. 
     
     Let $\sS'_1$ be a Picard-Vessiot algebra for $\sM_\sL$. It is simple, contains $\sS$, and its constant field is $C$ (since the constant field of $\sL$ is $C$ by condition (2) in Definition \ref{D3.1}).   
     
     Any object of $\langle \sM_\sK\rangle^{\otimes} $ is solvable in $\sS'_1$, whence a fiber a functor on the tannakian $C$-category 
   $\langle \sM_{C}\rangle^{\otimes} \cong \langle \sM_{\sK}\rangle^{\otimes} $ (\cf \ref{T2.1} (2)). 
The coordinate ring of the associated torsor of solutions is a Picard-Vessiot algebra $\sR'_1$ for   $\sM$  contained in $\sS'_1$. Since $\sR'_1$ contains $\sum_k\, \langle  {\rm{Sym}}^k \sM , ({\rm{Sym}}^k \sM^\vee_{\sS'_1})^\nabla\rangle$, it also contains $\sS $
by condition (3) in Definition \ref{D3.1}.  

   On the other hand, by Lemma \ref{L2.5}, $\sR'_1$ is isomorphic to $\sR'  $.

 (2)  According to \S \ref{su2.5}, $\omega(\sS)$ is a rational $G$-algebra of finite type over $C$. Let $v_1, \ldots, v_m$ be generators. The $G$-module $V_i$ generated by $v_i$ is of the form $\omega(\sN_i^\vee)$ for some $\sN_i\in  \langle \sM\rangle^{\otimes}  $. One has $\omega(\langle  \sN_i, v_i\rangle) =\langle \omega(  \sN_i), v_i\rangle = \langle V_i^\vee , v_i\rangle = V_i \subset \omega(\sS)$. Hence $v_i(\sN_i)= \langle  \sN_i, v_i\rangle \subset \sS$, and the image of the solution $v=\sum v_i$ of $\sN= \oplus \,\sN_i$ generates the $R$-algebra $S$. Since $Q(\sS)^\nabla\subset (\sK')^\nabla=C$, we conclude that $\sS$ is a solution algebra for $\sN$.  
 
 (3) Follows from the equivalence of categories established in item (2) of theorem \ref{T2.1}.  \end{proof} 
 
          \begin{ex}  If $\Omega= R$ and $\sM \cong \sR/ \sR. \phi$ is a cyclic differential module, then by item (2), a solution algebra for $\langle \sM\rangle^{\otimes} $ is the differential $R$-algebra generated by finitely many polynomials $P_j( y_i, y'_i, \ldots, 1/w)$ in solutions of $\phi$ (in some differential extension field with constant field $C$), their derivatives, and the inverse of the wronskian.
            \end{ex} 
   
   \subsection{}\label{Su3}  Let us further apply the considerations of \S \ref{su2.5} to solution algebras.   In the following theorem, ``solution algebra" means  ``solution algebra for some {$\sN\in  \langle \sM\rangle^{\otimes}\,$}."  They form a category (a full subcategory of the category of algebras in ${\rm{Ind}}\, \langle \sM\rangle^{\otimes}$).

   \begin{thm}\label{T3.2}    \begin{enumerate}   \item   $\sS \mapsto Z = \Spec \omega(\sS)$ gives rise to an anti-equivalence of categories between solution algebras for $\langle \sM\rangle^{\otimes} $ and affine quasi-homogeneous $G$-varieties.
           
 \smallskip  \item More precisely, it gives rise to a bijection between intermediate solution algebras $\,\sR\subset \sS\subset \sR'\,$ and pairs $\,(Z,v)\,$ (up to unique isomorphism) where $Z$ is an affine quasi-homogeneous $G$-variety and $v\in Z$ is a closed point of the dense orbit. 

          \smallskip    \item  Differential ideals of $\sS$ correspond to closed $G$-subsets  of $Z$.
        
        \smallskip    \item  For any solution algebra $\sS\subset \sR'$,  $R'$ is flat (and even smooth) over $S$. Moreover, $R'$ is faithfully flat over $S$ $\Leftrightarrow \sS$ is simple $\Leftrightarrow Z $ is a homogeneous $G$-variety.
      \end{enumerate}   \end{thm}
   
    \begin{proof}  
     (1) (2) If one embeds $\sS$ into the Picard-Vessiot algebra $\sR'$ (Proposition \ref{T3.1} (1)) and apply $\omega$ to the following morphisms of differential algebra extensions of $\sR$ in ${\rm{Ind}}\, \langle \sM\rangle^{\otimes}$: 
  $\;  {\rm{Sym}}^{\bf \cdot}\, \sM \stackrel{v^{\bf \cdot} }{\surj} \sS \inj \sR',\,$ one gets morphisms of rational $G$-algebras
  $$C[\omega(\sM^\vee)]= {\rm{Sym}}^{\bf \cdot}\, \omega(\sM) \stackrel{v^{\bf \cdot}}{\surj} \omega(\sS) \inj \omega(\sR')=C[G].$$ Identifying $v$ with a point in the vector space $V = \omega(\sM^\vee)$, the composed morphism $C[V]\to \omega(\sS) \inj C[G]$ is nothing but the comorphism of the morphism $G \to  V$ given by $g\mapsto g.v$, which factors through the dominant morphism  $\pi: G\to Z = \Spec\,\omega(\sS)$. It follows that the closed subset $Z $ of $V$ is the closure $\overline{G.v}\subset V$.  
  
  The $\otimes$-equivalence    $ \, {\rm{Ind}}\, \langle \sM\rangle^{\otimes} \stackrel{\omega}{\to} {\rm{Ind}}\, {\rm{Rep}}\, G\,$ thus induces a fully faithful contravariant functor from solution algebras $\sS$  for $\langle \sM\rangle^{\otimes} $ to affine quasi-homogeneous $G$-varieties $\,Z$, and an injection from intermediate solution algebras $\,\sR\subset \sS\subset \sR'\,$ to pairs $(Z, \pi(1))$.

  \smallskip Conversely, let $Z$ be an affine quasi-homogeneous $G$-variety, and $v\in Z$ be in the dense orbit, whence a dominant $G$-morphism $\,G\stackrel{\pi}{\to} Z = \Spec\,\omega(\sS),\; v= \pi(1)$.   
 Since $C[Z]$ is a rational $G$-algebra, it is a quotient of $\rm{Sym}^{\bf \cdot}\,V^\vee$ for some finite $G$-module $V$. This provides a closed $G$-embedding $Z\inj V$. Since $Z$ is quasi-homogeneous, it is the closure of a $G$-orbit $\,\overline{G.v}\in V$.   
   
   Let  $\,  \sN\in \langle \sM\rangle^{\otimes}\,$ be such that $\omega(\sN)= V^\vee$, let $\sS$ be the algebra in $\rm{Ind}\,\langle \sM  \rangle^{\otimes}$ such that $\omega(\sS )= C[Z]$, and let $v: \sN \to  \sS$ be the morphism whose image by $\omega$ is the given point $\,v\in V$.  Then $\rm{Sym}^{\bf \cdot}\,\sN\to \sS$ is an epimorphism since ${\rm{Sym}}^{\bf \cdot}\,V^\vee\to C[Z]$ is.  The choice of $v$ specifies the dominant $G$-morphism $G\to Z$, and corresponds via $\omega$ to an embedding $\sS\inj \sR'$. It follows that $S$ is a domain and that the field of constant of its quotient field is $C$. We conclude that $\sS$ is a solution algebra for $\langle \sM\rangle^{\otimes} $ generated by the image of the solution $v$.

 \smallskip  (3) is clear: $\; \sI \leftrightarrow \Spec\,\omega(\sS/\sI)$.

    (4) Applying the isomorphism \eqref{iso} to $\sN = \sS$ and $\sN=\sR'$,  smoothness (\resp faithful flatness) of $R'$ over $S'$ follows from smoothness (\resp is equivalent to faithful flatness) of $G\to Z$. By item (3), one has: $\sS$ is simple $\Leftrightarrow \overline{G.v}= G.v \Leftrightarrow G\to Z$ is faithfully flat. 
        \end{proof}

       \begin{rem} Any solution algebra $\sS$ is a domain by definition, but the associated quasi-homogeneous variety $\,Z= \overline{G.v}\,$ may be reducible. It may even occur that $ Z $ is connected but its dense orbit $\,G.v\,$ is disconnected, as the following example shows:    $\; \sM=  (C(z)^2, \nabla = d- \begin{pmatrix}0&1\\  \frac{1}{4z} &- \frac{1}{2z}      \end{pmatrix}),\; \sS= C(z)[e^{ \sqrt{z}}, \sqrt{z}\,e^{\sqrt{z}} ]\cong C(z)[x,y]/(y^2-zx^2)\subset \sR'= C(z)[e^{\pm\sqrt{z}}, \sqrt{z}],$ and $v$ sends the canonical basis of $M$ to $(e^{\sqrt z}, 0)$. Then $Z$ is the union of the axes in $\omega(\sM^\vee)= C^2,$ which are permuted by $\mu_2\subset G= \mathbb G_m \times \mu_2$.  
 
   This example also shows that, whereas $R'$ is always a smooth $S$-algebra, $ S$  {\it may not be a smooth  $ R$-algebra}.   \end{rem}

      \begin{rem}\label{RR3}  An integral quotient $\sS'= \sS/\sI$ of a solution algebra for $\langle \sM\rangle^{\otimes} $ is a solution algebra for $\langle \sM\rangle^{\otimes} $ if and only if the constant field of $Q(\sS')$ is $C$. This occurs if and only if the $G$-variety $\Spec\,\omega(\sS')$ is quasi-homogeneous. Such quotient solution algebras correspond exactly to $G$-orbits in $Z$.
      
      The question of finiteness of $G$-orbits is a classical problem in the study of quasi-homogeneous varieties (\cf \eg \cite{ArT2}\cite{Ar1} in the affine case). In the case of $Z$, this corresponds to the question of finiteness of quotient solution algebras of $\sS$.  \end{rem}

  \section{Solution fields and observable subgroups}
     
         \subsection{} Let $\sK$ be the quotient field of $\sR$ as in the previous section. 
         
         The quotient field $\sK'$ of $\sR'$ is a {\it Picard-Vessiot field} for $\sM_\sK$. It is minimal among the differential field extensions of $\sK$ with constant field $C$ in which $\sM_\sK$ and $\sM^\vee_{\sK}$ are solvable. The differential Galois group of $\sM_\sK$ (or $\sM$) is
          $G= {\rm{Aut}}\, \sK'/\sK.$
     
      The (generalized) {\it differential Galois correspondence} is an order-reversing bijection between intermediate differential extensions $\sK\subset \sL\subset \sK'$ and closed subgroups $H< G$, given by  $H = {\rm Aut}\, \sK'/\sL$ and $\sL= (\sK')^H$.
          Moreover $\sK'$ is a Picard-Vessiot field for $\sM_{\sL}$, and $\sL$ is a Picard-Vessiot field for some $\sN\in \langle \sM_\sK\rangle^{\otimes}$ is and only if $H\triangleleft G$,  \cf \cite[3.5.2.2]{A3}.

\begin{rem} Let $\,{\rm{Vec}}_{K,C}^{K'}$ be the category of triples $(P,W,\iota )$ where $P$ is a finite-dimensional $K$-vector space, $W$ is a finite-dimensional $C$-vector space and $\iota: W\otimes_CK'\to P\otimes_K K'$ is an isomorphism. This is actually a tannakian category over $C$. One has a $\otimes$-functor $\,\langle \sM_\sK\rangle^{\otimes}\to {\rm{Vec}}_{K,C}^{K'}$ which sends  $\sN_\sK$ to  $(P= N_K, W= (N_K\otimes_K K')^\nabla$, canonical isomorphism $\iota$). This makes $\,\langle \sM_\sK\rangle^{\otimes}$ a tannakian subcategory of ${\rm{Vec}}_{K,C}^{K'}$ (one easily checks that any subobject of $(N_K, (N_K\otimes_K K')^\nabla), \iota)$ comes from $\,\langle \sM_\sK\rangle^{\otimes}$).
\end{rem}

      \subsection{}    Let $\sL/\sK$ be a differential field extension, and let $v: \sM_\sK \to  \sL$ be a solution of $\sM_\sK$ in $\sL$ (\ie a morphism of differential modules). 
     
         \begin{defn}\label{D4.1}  $\sL$ is a {\emph{solution field}} for $\sM_\sK$ if its constant field is $C$ and there is a morphism $\sM_\sK\to \sL$ of differential modules over $\sK$ whose image generates the field extension $L/K$.
         
         A solution field for $\langle \sM_\sK\rangle^{\otimes} $ is a solution field for some $\sN_\sK \in \langle \sM_\sK\rangle^{\otimes} $.
         \end{defn} 

\begin{lemma}\label{L4.2}  \begin{enumerate}  \item  The quotient field of a solution algebra $\sS$ for $\sM$ is a solution field for $\sM_\sK$.   \smallskip\item  Conversely, any solution field $\sL$ for $\sM_\sK$ is the quotient field of a (non unique) solution algebra $\sS$ for $\sM$. 
   \end{enumerate}  
\end{lemma}  
\begin{proof} (1) is immediate.
For (2), let $S$ be the $R$-subalgebra of $L$ generated by $v(M)$. It is clear that this is a differential algebra with quotient field $\sL$, and the conditions for a solution algebra are satisfied.
\end{proof}
   
   \begin{thm}\label{T4.1}   Let $\sK'/\sK$ be a Picard-Vessiot field for $\sM_\sK$.
   \begin{enumerate}  \item Any solution field $\sL$ for $\langle \sM_\sK\rangle^{\otimes} $ embeds as a differential sub-extension of $\sK'/\sK$. 
    \smallskip \item If $\sL\subset \sK'$ is the quotient field of a solution algebra $\sS$ for $\langle \sM_\sK\rangle^{\otimes} $, then $\sS\subset \sR'$.
  \smallskip \item  An intermediate differential field $\,\sK\subset \sL\subset \sK'\,$ is a solution field for $\langle \sM_\sK\rangle^{\otimes} $ if and only if $\,H= {\rm{Aut}}\,\sK'/\sL$ is an \emph{observable subgroup} of $\,G={\rm{Aut}}\,\sK'/\sK$.  
  
  In fact, $H$ is the isotropy group of any solution $v: \sN_\sK \to \sL$ whose image generates $\sL$.
   \smallskip \item  For any solution field $\sL = (\sK')^H$ for $\langle \sM_\sK\rangle^{\otimes} $,   $\;\;N_G(H)/H ={\rm{Aut}}\,\sL/\sK $.
     \end{enumerate}    \end{thm}
   
There are many equivalent characterizations of observable subgroups $H<G$, \cf \cite[Th. 2.1]{G}. One  is that $G/H$ is quasi-affine. Another is that every finite-dimensional rational $H$-module extends to a finite-dimensional rational $G$-module. A third one is that $H$ is the isotropy group of a vector $v$ in some rational $G$-module (and one may even require that $H$ is also the stabilizer of the line $Cv$, \cf \cite{N}). Recall also that $G$ is observable if it has no non-trivial rational character.
      
   \begin{proof} (1) is a consequence of Proposition \ref{T3.1} via item (2) of Lemma \ref{L4.2}.
      
   (2) Let $\iota_1$ be the given embedding $\sL \to \sK'$. By Proposition \ref{T3.1} again, there is an embedding $  \sS\to \sR'$, which gives rise to a second embedding $\iota_2: \sL\to \sK'$. Since $\sK'$ is a Picard-Vessiot field for $\sM_{\sL}$ with automorphism group $H$, $\iota_1= h\circ \iota_2$ for some $h\in H \subset G$. Since $G$ preserves $\sR'$ and $\iota_2(\sS)\subset \sR'$, one has $\iota_1(\sS)\subset \sR'$.

 \smallskip  In (3) and (4), one may replace $\sR$ by its quotient field $\sK$ (taking into account item (3) of Proposition  \ref{T3.1}).
   
  \smallskip (3) Let $V$ be a finite-dimensional $G$-module, and $H$ be the isotropy group of a vector $v\in V$. Let us write $V= \omega(\sN^\vee)$ for some $\sN\in  \langle \sM\rangle^{\otimes} $. Then $(K')^H$ is the subfield of $K'$ generated by $\langle N, v\rangle$.
   
   Indeed, let $H<H'<G$ be the intermediate group attached to this subfield. Then for any $n\in N$ and any $h\in H',  \, \langle n, h.v\rangle = h(\langle n,  v\rangle)= \langle n,  v\rangle$, and one concludes that $h.v=v$, whence $H=H'$.
   
   Now, any observable subgroup $H$ is such an isotropy group, and the previous observation shows that $L= (K')^H$ is a solution field generated by $v$. 
     Conversely, if $L$ is a function field generated by a solution $v$ of $\sN\in     \langle \sM\rangle^{\otimes} $, and $H'$ is the subgroup attached to  $L= (K')^{H'}$, the previous observation shows that $H'$ coincides with the isotropy group $H$ of $v$ in $\omega(\sN^\vee)$, hence is observable.
      
   (4)  
   One has $\omega((\sR')^H) = C[G]^H= C[G/H]$, hence ${\rm{Aut}}\,(\sR')^H/\sK= {\rm{Aut}}_G\,\omega((\sR')^H)= {\rm{Aut}}_G\, C[G/H] = {\rm{Aut}}_G\,  G/H  = N_G(H)/H$ (acting on $G/H$ by $nH\cdot gH = gn^{-1}H$). 
   
   Note that $\sL$ is the quotient field of $\sL \cap \sR' = (\sR')^H$ (this follows from item (2) above and the previous lemma); hence $ {\rm{Aut}}\,(\sR')^H/\sK \subset{\rm{Aut}}\,\sL/\sK $. It remains to show that any automorphism of $\sL$ preserves $  (\sR')^H$. One observes that ${\rm{Aut}}\,\sL/\sK$ permutes the differential subalgebras of $\sL$ which are finitely generated over $K$,  hence preserves their union. This union is contained in $(\sR')^H$, in fact equal to it since it is an algebra in   ${\rm{Ind}}\,  \langle \sM\rangle^{\otimes} $.  \end{proof}
   
          \begin{rem} ${\rm{Aut}}\,\sS/\sR$ may be smaller than ${\rm{Aut}}\,\sL/\sK$. Equality occurs precisely when the corresponding quasi-homogeneous variety $\overline{G.v}$ is {\it very symmetric} in the sense of \cite[\S 4.3]{Ar1}, \cf also \cite[\S 2]{ArT2} (this is the case whenever $H$ is a {\it spherical} observable subgroup of a connected reductive group $G$). 
          
          On the other hand, ${\rm{Aut}}\,\sR'/\sS $ coincides with $ {\rm{Aut}}\,\sK'/\sL = H$ since $H$ preserves $\sR'$ and $\sL$ is the quotient field of $\sS$.
  \end{rem}

     \section{Homogeneous solution algebras}\label{su5}
 
 \subsection{} Let $\sS$ be a solution algebra generated by a solution $v: \sM\to \sS$, and let $v^{\cdot}$ be its canonical extension to a surjective homomorphism of differential rings $ {\rm{Sym}}^{\cdot} \sM \to \sS$.     
  Let ${\tilde \sS}$ be the quotient of $ {\rm{Sym}}^{\cdot} \sM$ by the graded ideal $I$ 
 generated by homogeneous relations in ${\rm{Ker}}\, v^{\cdot}$, which is clearly a differential ideal: $$\, \sI= \oplus \,\sI^i, \; \sI^i= \Ker ({\rm{Sym}}^i \sM \to \sS),\; \tilde S = \oplus \,\tilde S^i,\; \tilde S^i = ({\rm{Sym}}^i\sM) /\sI^i \inj S.$$ 
  
  We first observe that, like $S$, ${\tilde  S}$ {\it is a domain}: if $  a,  b\in \tilde S$ have homogeneous decompositions $ \sum a_i$ and $\sum b_i$ respectively, and satisfy $\, a.   b= 0$, then the product of $ \sum a_i t^i$ and $ \sum b_i t^i$ must be $0$ in $\oplus\, \tilde S^i t^i\subset \tilde S[t]$ (since $\tilde S$ is a graded ring), hence goes to $0$ in $S[t]$. Since $S[t]$ is a domain, and $\oplus\, \tilde S^i t^i$ maps injectively into $S[t]$, we conclude that $  a =0$ or $ b = 0$. 
  
  On the other hand, $\tilde S \in {\rm{Ind}}\, \langle \sM\rangle^{\otimes}  $, hence is faithfully flat over $R$ by Corollary \ref{C1.4}.
    Thus  ${\rm{Proj}}\, {\tilde S}$ is an integral closed subscheme of ${\mathbb P}(M)$, faithfully flat over $R$.
     
 \subsection{} Note that  $\omega(\tilde S)$ is a graded $G$-algebra, and ${\rm{Proj}}\, {\tilde S}$ is a closed $G$-subvariety of the projective space $\,{\mathbb P}(\omega(\sM)) $ of lines in $\,V =\omega(\sM^\vee)$, which contains the image $\tilde v = [Cv]\in \,{\mathbb P}(\omega(\sM)) $ of $v\in  V$.
         Let  $\tilde H$ be the isotropy group of $\tilde v$ in $G$. The isotropy group $H$ of $v$ is normal in $\tilde H$ and the quotient $\tilde H/H$ is a closed subgroup of $\mathbb G_m$. 
       
         If $S= \tilde S$, one has a commutative square 
       \[\begin{CD}
G/H @>>>   ( \Spec \omega(\sS) )\setminus 0 \\
 @VVV  @VVV\\
   G/\tilde H@>>> {\rm{Proj}}\,\omega(\sS).
\end{CD}\]
Since the horizontal morphisms are immersions, the top one being open, and since the right vertical morphism is the quotient map by $\mathbb G_m$, one must have $\tilde H/H \cong\mathbb G_m.$        
       
  Conversely, assume that $\tilde H/H \cong\mathbb G_m.$ It can be considered as a closed subgroup of $N_G(H)/H= {\rm{Aut}}\,\sL/\sK$ (Th. \ref{T4.1} (4)). Denoting by $t\ast \ell$ the action of $t\in  C^\ast$ on $\ell\in \sL$, one has $\, t\ast (v^i(n))= (t^iv^i)(n), $ for any $i\geq 0$ and any $n\in {\rm{Sym}}^i\sM,$ so that the action $\ast$ induces a graduation of $\sS$ compatible with  ${\rm{Sym}}^{\cdot} M \to S$. This means that $S=\tilde S$. 
      
In that case, ${\rm{Proj}}\,\omega(\sS)$ is a projective quasi-homogeneous $G$-variety: indeed, in the above commutative diagram, the top and right morphisms are dominant, hence the bottom morphism is dominant as well. 

\begin{rem} This situation occurs for instance when $H$ is a {\it quasi-parabolic subgroup} of $G$, \ie the isotropy subgroup of a highest weight vector in some irreducible $G$-module. In that case, the horizontal maps of the above commutative diagram are isomorphisms  (\cf \cite{PV}).\end{rem}

   \section{Proof of the statements of \S 1} These statements concern classical differential rings (\ie the case $\Omega=R$), but extend to the case of generalized differential rings, where $\Omega$ is any projective $R$-module of finite rank. 
   
   \subsection{} Theorem \ref{T1.1}  follows immediately from Theorem \ref{T4.1}. 
   
     \subsection{Proof of Theorem \ref{T1.2}}\label{PT1.2} (1) follows from Proposition \ref{T3.1} (2). 
     
     (2) follows from Theorem \ref{T4.1} (2). 
     
     (3) follows from Theorem \ref{T3.2} (1).   
          
     (4) follows from the fact that $\sR'^H\in \rm{Ind}\, \langle \sM\rangle^{\otimes} $ corresponds via $\omega$ to $C[G]^H=C[G/H]$. Hence $\sR'^H$ (which is the maximal localization $Q(\sS)\cap \sR'$ of $\sS$ in $\sR'$) generated by some object in $\langle \sM\rangle^{\otimes} $ if and only if $C[G/H]$ is generated by a finite $G$-module, which amounts to saying that $H$ is Grosshans.
       
    (5) follows from Theorem \ref{T3.2}(4) (note that if $\sS$ is simple, $G/H$ is affine, hence is the spectrum of $ C[G/H] =   \omega(\sR')^H=  \omega(\sR'^H)$.
     
   (6):  let $\sL= (\sK')^H$ be a solution field for $\langle \sM_\sK\rangle^{\otimes} $. Then  $L$ is the quotient field of a {{unique}} solution algebra $S$ (necessarily contained in $\sR'^H$) if and only if there is a unique affine quasi-homogeneous variety $Z$ with dense orbit $G/H$ (hence $Z = G/H$). In the terminology of invariant theory, $G/H$ is affinely closed. According to Luna  \cite{L}, in case $G$ is reductive, and to Arzhantsev and Timashev \cite[\S 3.3]{ArT2} in general, this occurs precisely when the image $\bar H$  of $H$ in the reductive quotient $\bar G$ of $G$ is reductive and $N_{\bar G}(\bar H)/\bar H$ is finite.  
       
    (7) follows from Corollary \ref{C1.5}.

        \subsection{Proof of Theorem \ref{T1.3}} (1) $\sM$ is semisimple if and only if $G$ is reductive. For any $W\in {\rm{Rep}}\,G$ such that the action of $G$ factors through a finite group $G'$, the corresponding Picard-Vessiot algebra is a finite connected torsor under $G' $ over $C[z]$, hence $G'=\{1\}$. Therefore  $G$ is connected. According to Raghunathan and Ramanathan \cite{RR}, any torsor under a connected reductive group over $C[z]$ is trivial, hence the torsor of solutions of $\sM$ is trivial, which means that $\omega_{C[z]}\cong \vartheta$ (\cf \S \ref{su2.3}). In particular, $\omega(\sS)_{C[z]}\cong S$ as $R$-algebras, and $Z=\Spec\,\omega(\sS)$ is a quasi-homogeneous $G$-variety by Theorem \ref{T3.2} (1).
        
        (2) Let $G$ be connected reductive, and let $Z$ be an affine quasi-homogeneous $G$-variety. As in the proof of  \ref{T3.2} (2), one can embed $Z$ as a closed $G$-subset in a finite-dimensional $G$-module $V$ (which we may assume to be faithful).   The constructive solution (by Mitschi and Singer \cite{MS}) of inverse differential Galois theory attaches to $G\inj GL(V)$ a (semisimple) differential module $\sM$ over $C[z]$ with differential Galois group $G$. Theorem \ref{T3.2} (1)  shows how to construct a solution algebra $\sS$ for $\sM$, with $\omega(\sS)= C[Z]$, and by the previous item, $\omega(\sS)_{C[z]}\cong S$ as $R$-algebras.

             \subsection{Proof of Theorem \ref{T1.4}} (1) has been proven in \S \ref{su5}.
                          
             (2) follows from Corollary \ref{C1.5}.     
             
             (3) Since $\Spec S$ is an algebraic fiber bundle over $\Spec R$, all fibers are integral if and only if the generic fiber is geometrically integral, \ie $K$ is algebraically closed in $L=Q(S)$. Assume that this is the case. 
             
             Since $\rm{Proj}\, \tilde S$ is an algebraic fiber bundle over $\Spec R$, all fibers are integral if and only if the generic fiber of the affine cone is geometrically integral. One may assume that $R=K$, and one has to show that for any finite extension $\sK_1/\sK$ in $\tilde \sS$, $\tilde S\otimes_K {K_1}$ is a domain. This is done by the same argument as in \S \ref{su5}, taking into account the fact that $(S\otimes_K {K_1})[t] $ is a domain. 
             
      \subsection{Final remarks}\label{R6.1}  (1) In the context of Corollary \ref{C1.5}, one can deduce  directly the homogeneous case from the inhomogeneous case, as follows. Let  $P(y , \ldots , y^{(n-1)})=0$ be a polynomial relation of degree $D$ with coefficients in $R$, which becomes homogeneous of degree $d\leq D$ after specialization at $z=\xi$.   Let $P_d$ be the homogeneous part of degree $d$ of $P$, and write $P= P_d+(z-\xi)Q$. Then $Q$ (\resp $P_d$) maps naturally to an element  of $S^{\leq D}= \im(\tilde S^{\leq D}\to S)$ (\resp $S^d = \im(\tilde S^{d}\to S)$). The quotient $S^{\leq D}/S^d$ is a finitely generated differential $R$-module, hence torsion-free since $\sR$ is simple. Since $(z-\xi)Q$ goes to $0$ in $S^{\leq D}/S^d$, so does $Q$, \ie there is $Q_d$ homogeneous of degree $d$ such that $(P_d+(z-\xi)Q_d)(y , \ldots , y^{(n-1)})=0$.

   \smallskip (2) One question frequently asked by algebraic geometers regarding differential Galois theory is the following: is there a ``sheaf-theoretic version" valid over any smooth connected algebraic $C$-variety $X$ (not necessarily affine)?  
             Here is an answer.
             
             The generalized differential ring $\sR$ is replaced by $(X,\,  d_X: \sO_X \to \Omega^1_X)$. Being in characteristic $0$ ensures that $\Ker d_X$ is the constant sheaf $C$. Differential extensions $\sS/\sR$ have to be replaced by (not necessarily smooth) morphisms  $Y\stackrel{f}{\to} X $  together with a retraction $\rho: \Omega^1_Y\to   f^\ast\Omega^1_X$ of the natural morphism $ f^\ast\Omega^1_X\to \Omega^1_Y$ (assumed to be injective); whence a derivation  $d= \rho\circ d_Y: \sO_Y \to f^\ast\Omega^1_X$  extending $f^{-1}d_X$. 
             
             Let $\sM$ be a coherent $\sO_X$-module with a (not necessarily integrable) connection.  The underlying module is locally free and the category of subquotients of finite direct sums of $ \sM^{\otimes i}\otimes (\sM^\vee)^{\otimes j}$ is neutral tannakian over $C$. The fiber at any closed point $x$ is a fiber functor $\omega_x$ with values in ${\rm{Vec}}_C$. The differential Galois group pointed at $x$ is $G_x={\bf{Aut}}^\otimes \omega_x$. One constructs the torsor of solutions $\Sigma_x$ as in the affine case; it is a torsor under the affine $X$-group $(G_x)_X$, and it admits a canonical structure of differential extension in the above sense.  All this is a straightforward modification of \S\ref{su2.2}, \ref{su2.3}, \ref{su2.4}. 
      
       \smallskip      (3) We expect that a similar theory of solution algebras holds in characteristic $p$, provided one uses Schmidt ``iterated derivatives" or (in higher dimension) the ring of differential operators in the sense of Grothendieck [EGAIV, \S 16.8]. 
          
          We also expect a similar theory for difference equations, or mixed difference-differential equations (for instance $p$-adic differential equations with Frobenius structure),  and we even expect a common framework with the above theory, using non-commutative bimodules $\Omega$ as in \cite{A3}, which unifies differential algebra and difference algebra. One should however pay attention to the fact that simple difference rings may have zero divisors. In the definition of (difference) solution algebras, one should then replace the condition that $\sS$ is a domain by the condition that it be contained in a simple difference algebra.  
    
    \end{sloppypar}

\bigskip{\smit Acknowledgements.} {\sm I thank A. Pianzola for several useful discussions about torsors on open subsets of the line (\cf \ref{RR2}), and S. Gorchinsky for a remark which led to a simplification of the proof of Proposition \ref{T3.1}.

 \end{document}